\documentclass[11pt]{amsproc}

\theoremstyle{definition}

\theoremstyle{remark}

\numberwithin{equation}{section}
\usepackage{graphicx}
\usepackage{latexsym}
\usepackage{amsmath}
\usepackage{amssymb}
\usepackage{amsfonts}
\usepackage{verbatim}
\usepackage{mathrsfs}
\usepackage[modulo]{lineno} 
\usepackage[latin1]{inputenc}
\usepackage{color}
\usepackage[colorlinks,citecolor=blue,urlcolor=blue]{hyperref}

\newtheorem{tm}{Theorem}[section]
\newtheorem{rk}{Remark}[section]
\newtheorem{ap}{Assumption}[section]

\newtheorem{prop}{Proposition}[section]
\newtheorem{lm}{Lemma}[section]
\newtheorem{cor}{Corollary}[section]
\newtheorem{ex}{Example}[section]

\newcommand{\ee}{\mathbb E}
\newcommand{\ff}{\mathbb F}
\newcommand{\ii}{\mathbb I}
\newcommand{\pp}{\mathbb P}
\newcommand{\nn}{\mathbb N}
\newcommand{\rr}{\mathbb R}

\newcommand{\BB}{\mathcal B}
\newcommand{\CC}{\mathcal C}
\newcommand{\LL}{\mathcal L}
\newcommand{\TT}{\mathcal T}

\newcommand{\PP}{\mathcal P}

\newcommand{\OOO}{\mathscr O}
\newcommand{\FFF}{\mathscr F}

\newcommand{\<}{\langle}
\renewcommand{\>}{\rangle}
\allowdisplaybreaks \allowdisplaybreaks[4]

\begin{document}


\title[Numerical Ergodicity and Uniform Estimate of Monotone SPDE]
{Numerical Ergodicity and Uniform Estimate of Monotone SPDEs Driven by Multiplicative Noise}
 
\author{Zhihui Liu}
\address{Department of Mathematics \& National Center for Applied Mathematics Shenzhen (NCAMS) \& Shenzhen International Center for Mathematics, Southern University of Science and Technology, Shenzhen 518055, China}
\curraddr{}
\email{liuzh3@sustech.edu.cn}
\thanks{The author is supported by the National Natural Science Foundation of China, No. 12101296, Basic and Applied Basic Research Foundation of Guangdong Province, No. 2024A1515012348, and Shenzhen Basic Research Special Project (Natural Science Foundation) Basic Research (General Project), No. JCYJ20220530112814033 and No. JCYJ20240813094919026.} 



\begin{abstract}
We analyze the long-time behavior of numerical schemes for a class of monotone stochastic partial differential equations (SPDEs) driven by multiplicative noise.    
By deriving several time-independent a priori estimates for the numerical solutions, combined with the ergodic theory of Markov processes, we establish the exponential ergodicity of these schemes with a unique invariant measure, respectively.
Applying these results to the stochastic Allen--Cahn equation indicates that these schemes always have at least one invariant measure, respectively, and converge strongly to the exact solution with sharp time-independent rates.
We also show that these numerical invariant measures are exponentially ergodic and thus give an affirmative answer to a question proposed in (J. Cui, J. Hong, and L. Sun, Stochastic Process. Appl. (2021): 55--93), provided that the interface thickness is not too small. 
\end{abstract}

\maketitle



\section{Introduction}

Recently, a lot of researchers investigated numerical analysis of the following SPDE under the homogeneous Dirichlet boundary condition in a finite time horizon:
\begin{align}\label{see-fg}
\begin{split}
& {\rm d} X(t, \xi)  
=(\Delta X(t, \xi)+f(X(t, \xi))) {\rm d}t
+g(X(t, \xi)) {\rm d}W(t, \xi), \\
& X(t, \xi)=0,\quad (t, \xi) \in \rr_+\times \partial \OOO; \\
& X(0, \xi)=X_0(\xi), \quad \xi \in \OOO,
\end{split}
\end{align}
where the physical domain $\OOO \subset \rr^d$ ($d=1,2,3$) is a bounded open set  with smooth boundary. 
Here, $f$ is assumed to be of monotone type with polynomial growth, $g$ satisfies the usual Lipschitz condition in an infinite-dimensional setting, and $W$ is an infinite-dimensional Wiener process, see Section \ref{sec2.1}.
Eq. \eqref{see-fg} includes the following semiclassical stochastic Allen--Cahn equation, arising from phase transition in materials science by stochastic perturbation,  as a special case:
\begin{align} \label{ac}
{\rm d} X=  \Delta X  {\rm d}t + \alpha^{-2} (X-X^3) {\rm d}t
+ {\rm d}W,
\quad X(0)=X_0,
\end{align}
where $\alpha>0$ is the interface thickness; see, e.g., \cite{BGJK23}, \cite{BJ19}, \cite{BCH19},  \cite{CH19}, \cite{GM09}, \cite{LQ19}, \cite{LQ20}, \cite{LQ21}, \cite{MP18} and references therein. 
 
The numerical behavior in the infinite time horizon, especially the numerical ergodicity of  SPDEs, including Eq. \eqref{see-fg} (and Eq. \eqref{ac}), is a natural and intriguing question.
As a significant asymptotic behavior, the ergodicity characterizes the case of temporal average coinciding with spatial average, which has vital applications in quantum mechanics, fluid dynamics, financial mathematics, and many other fields \cite{DZ96}, \cite{HW19}.
The spatial average, i.e., the mean of a given function for the stationary law, called the invariant measure of diffusion, is also known as the ergodic limit, which is desirable to compute in many practical applications.
Then, one has to investigate a stochastic system over long time intervals, which is one of the main difficulties from the computational perspective.  
Generally, the explicit expression of the invariant measure for a nonlinear infinite-dimensional stochastic system is rarely available.
Therefore, it is usually impossible to precisely compute the ergodic limit for a nonlinear SPDE driven by multiplicative noise; exceptional examples are gradient Langevin systems (driven by additive noise), see, e.g., \cite{BV10}, \cite{LV22}.
For this reason, it motivated and fascinated a lot of investigations in the recent decade for constructing numerical algorithms that can inherit the ergodicity of the original system and approximate the ergodic limit efficiently.

In finite-dimensional case, much progress has been made in the design and analysis of numerical approximations of the desired ergodic limits; see, e.g., \cite{MSH02}, \cite{WL19}, \cite{LMW23} and references therein for numerical ergodicity of stochastic
ordinary differential equations (SODEs).
On the contrary, the construction and analysis of numerical ergodic limits for SPDEs are still in their early stages.
The authors in \cite{Bre14}, \cite{BK17}, \cite{BV16} firstly studied Galerkin-based linear implicit Euler scheme and high order integrator to approximate the invariant measures of a parabolic SPDE driven by additive white noise; see also \cite{CHW17} for the ergodic limit of a spectral Galerkin modified implicit Euler scheme of the damping stochastic Schr\"odinger equation driven by additive trace-class noise.  
Recently, exponential Euler-type schemes of full discretization and tamed semi-discretization were used in \cite{CGW20} and \cite{Bre22}, respectively, to approximate the invariant measure of a stochastic Allen--Cahn type equation driven by additive colored noise.
The authors in \cite{CHS21} investigated the spectral Galerkin drift implicit Euler (DIE) scheme to approximate the invariant measure of the same type of equation and proposed a question of whether the invariant measure of the temporal DIE scheme is unique.   
Almost all of the above literature focuses on the numerical ergodicity of SPDEs driven by additive noise; the numerical ergodicity in the multiplicative noise case is more subtle and challenging. 

Apart from the approximation of the invariant measure, the strong approximation, i.e., in the moments' sense, is also an important issue. 
There exists a general theory of strong error analysis of numerical approximations for Lipschitz SPDEs in a finite time horizon; see, e.g., \cite{ACLW16}, \cite{CHL17}, \cite{JR15} and references therein.
Recently, we \cite{LQ21} give a theory of strong error analysis of numerical approximations for monotone SPDEs in a finite time horizon.
Whether there is a theory of strong error analysis in the infinite time horizon is a natural question.   

The above two questions on long-time behaviors of numerical approximations for Eq. \eqref{see-fg} motivate the present study. 
On the one hand, we aim to establish the exponential ergodicity of the numerical schemes studied in \cite{GM09}, \cite{LQ21}, \cite{MP18} for Eq. \eqref{see-fg}. 
It should be pointed out that there is no further restriction on the time step size for the DIE scheme \eqref{die} in addition to the usual requirement of the implicit scheme (see Theorem \ref{tm-erg-uk}); see \cite{LMW23} for an upper bound requirement on the time step size for numerical ergodicity of monotone SODEs. 
Meanwhile, applying one of our main results, Theorem \ref{tm-erg-uk}, we conclude that the DIE scheme \eqref{die} is unique ergodic (and exponentially mixing) and thus give an affirmative answer to the question proposed in \cite{CHS21} (see Remark \ref{rk-solution}), provided that the monotone coefficient of the drift function is not too large. 
On the other hand, we develop a time-independent strong error analysis theory for these fully discrete schemes, see Theorem \ref{tm-err}, which generalizes the main results of \cite{LQ21} to the infinite time horizon. 

The paper is organized as follows.
In Section \ref{sec2}, we give several necessary uniform estimates and exponential ergodicity of Eq. \eqref{see-fg}.
These uniform estimates are utilized in Section \ref{sec3} to establish the exponential ergodicity of the proposed semi-discrete and fully discrete schemes.
In Section \ref{sec3}, we also derive time-independent strong error estimates for those above fully discrete schemes.
Finally, the main results in Section \ref{sec3} are applied to the stochastic Allen--Cahn equation \eqref{ac} in Section \ref{sec4}.

\section{Uniform A Priori Estimates and Exponential Ergodicity of Monotone SPDEs}
\label{sec2}

In this section, we first present the required preliminaries and main assumptions that will be used throughout the paper.
Then, we derive several necessary uniform a priori estimates of Eq. \eqref{see-fg}.

\subsection{Main Assumptions}
\label{sec2.1}
  
Denote $a \vee b:=\max(a, b)$ and $a \wedge b:=\min(a, b)$ for $a, b \in \rr$.
For $r \in [2, \infty]$ and $\theta=-1$ or $1$, we use the notations $(L^r=L^r(\OOO), \|\cdot\|_{L^r})$ and $(\dot H^\theta=\dot H^\theta(\OOO), \|\cdot\|_\theta)$ to denote the usual Lebesgue and Sobolev interpolation spaces, respectively.
In particular, we denote $H:=L^2$, $V:=\dot H^1$, and $V^*:=\dot H^{-1}$.
The space $H$ is equipped with Borel $\sigma$-algebra $\BB(H)$ and with the inner product and norm given by $\|\cdot\|$ and $\<\cdot, \cdot\>$, respectively.
The norms in $\dot H^1$ and $\dot H^{-1}=(\dot H^1)^*$ and the duality product between them are denoted by $\|\cdot\|_1$, $\|\cdot\|_{-1}$,  and $_1\<\cdot, \cdot\>_{-1}$, respectively. 
We use $\CC_b(H)$ and ${\rm Lip}_b(H)$ to denote the class of bounded, continuous, and Lipschitz continuous, respectively, functions on $H$.  
For $\phi \in {\rm Lip}_b(H)$, define 
$|\phi|_{\rm Lip}:=\sup_{x \neq y} |\phi(x)-\phi(y)|/\|x-y\|$.

Let ${\bf Q}$ be a self-adjoint and positive definite linear operator on $H$. 
Denote $(H_0:=Q^{1/2} H, \|\cdot\|_{H_0}:=\|Q^{-1/2} \cdot \|)$ and denote by
$(\LL_2^0:=HS(H_0; H), \|\cdot\|_{\LL_2^0})$ and $(\LL_2^1:=HS(H_0; \dot H^1), \|\cdot\|_{\LL_2^1})$ the space of Hilbert--Schmidt operators from $H_0$ to $H$ and $\dot H^1$, respectively.     
Let $W$ be an $H$-valued {\bf Q}-Wiener process on a filtered probability space $(\Omega,\FFF, \ff:=(\FFF_t)_{t \ge 0}, \pp)$, which has the Karhunen--Lo\`eve expansion 
$$W(t)  =\sum_{k\in \nn} \sqrt{q_k} g_k \beta_k(t), \quad t \ge 0.$$
Here $\{g_k\}_{k=1}^\infty$ forming an orthonormal basis of $H$ are the eigenvectors of {\bf Q} subject to the eigenvalues $\{q_k\}_{k=1}^\infty$, and $\{\beta_k\}_{k=1}^\infty $ are mutually independent one-dimensional Brownian motions in $(\Omega, \FFF, \ff, \pp)$.
We mainly focus on trace-class noise, i.e., {\bf Q} is a trace-class operator, or equivalently, ${\rm Trace}({\bf Q}):=\sum_{k\in \nn} q_k<\infty$.

Our main conditions on the coefficients of Eq. \eqref{see-fg} are the following two assumptions.

\begin{ap} \label{ap-f}
$f:\rr\rightarrow \rr$ is differentiable and there exist scalars $L_i \in \rr$, $i=1,2,3,4,5$, and $q \ge 1$ such that 
\begin{align} 
& (f(\xi)-f(\eta)) (\xi-\eta) \le L_1 (\xi-\eta)^2,
\quad \xi, \eta \in \rr, \label{f-mon} \\
& f(\xi) \xi \le L_2 |\xi|^2 + L_3,
\quad \xi \in \rr,  \label{f-coe} \\
& |f'(\xi)|\le L_4 |\xi|^{q-1}+L_5,\quad \xi \in \rr. \label{con-f'}
\end{align}
\end{ap}

\begin{ex} \label{rk-ex-f}
A concrete example such that all the conditions in Assumption \ref{ap-f} holds is an odd polynomial (with order $q$ and) with a negative leading coefficient. 
\end{ex}

By the condition \eqref{f-mon} and the differentiability of $f$, it is clear that
\begin{align} \label{f'}
f'(\xi) \le L_1, \quad \xi \in \rr.
\end{align}  
We also note that one can always assume that $L_1 \le L_2$, as one can take $L_2=L_1+\varepsilon$ in \eqref{f-coe} from \eqref{f-mon} and Young inequality; here and in the rest of the paper, $\varepsilon$ denotes an arbitrarily small positive constant which would be different in each appearance. 
 
Define by $F: L^{q+1} \rightarrow L^{(q+1)'}$ the Nemytskii operator associated with $f$:
\begin{align} \label{df-F}
F(x)(\xi):=f(x(\xi)),\quad x \in \dot H^1,\ \xi \in \OOO,
\end{align} 
where $(q+1)'$ denotes the conjugation of $q+1$, i.e., $1/(q+1)'+1/(q+1)=1$.
Then it follows from the monotonicity condition \eqref{f-mon} and the coercivity condition \eqref{f-coe} that the operator $F$ defined in \eqref{df-F} has a continuous extension from $L^{q+1}$ to $L^{(q+1)'}$ and satisfies
\begin{align} 
_{L^{q+1}}\<x-y, F(x)-F(y)\>_{L^{(q+1)'}} \le L_1 \|x-y\|^2,
& \quad  x, y \in L^{q+1}, \label{F-mon} \\
_{L^{q+1}}\<x, F(x)\>_{L^{(q+1)'}} \le L_2 \|x\|^2+L_3,
& \quad  x \in L^{q+1}, \label{F-coe}
\end{align}
where $_{L^{(q+1)'}} \<\cdot, \cdot\>_{L^{q+1}}$ denotes the duality product between $L^{(q+1)'}$ and $L^{q+1}$.

Throughout, we assume that $q \ge 1$ when $d=1,2$ and $q \in [1,3]$ when $d=3$, 
so that the following frequently used Sobolev embeddings hold:
\begin{align} \label{emb} 
\dot H^1 \subset L^{2q} \subset  L^{q+1} \subset H \subset L^{(q+1)'}\subset L^{(2q)'}  \subset \dot H^{-1}.
\end{align} 
In this case, \eqref{F-mon} and \eqref{F-coe} can be expressed as
\begin{align} 
_1\<x-y, F(x)-F(y)\>_{-1}  \le L_1 \|x-y\|^2, & \quad x,y \in \dot H^1, \label{F-mon+}  \\
_1\<x, F(x)\>_{-1} \le L_2 \|x\|^2+L_3, & \quad x \in \dot H^1. \label{F-coe+} 
\end{align}
The inequalities \eqref{F-mon+} and \eqref{F-coe+}, in combination with the Poincar\'e inequality that 
\begin{align} \label{poin}
\|\nabla x\|^2 \ge \lambda_1 \|x\|^2, \quad x \in \dot H^1, 
\end{align}
where $\lambda_1$ denotes the first eigenvalue of the Dirichlet Laplacian operator, denoted by $A$ with domain ${\rm Dom}(A)=H^2 \cap \dot H^1$, implies that 
\begin{align} 
& _1\<x-y, A(x-y)+F(x)-F(y)\>_{-1} \nonumber  \\ 
& \le -(\lambda_1-L_1) \|x-y\|^2, \quad x, y \in \dot H^1, \label{F-mon-0}  
\end{align}
and that
\begin{align} 
_1\<x, A x+F(x)\>_{-1} 
& \le -(\lambda_1-L_2) \|x\|^2 +L_3, \quad x \in \dot H^1. \label{F-coe-0}
\end{align}
Similarly, noting that $L_1 \le L_2$, we have
\begin{align} \label{F-coe-1}
\<x, A x+F(x)\>_1 
& \le -(\lambda_1-L_2) \|x\|_1^2 +L_3, \quad x \in \dot H^2.
\end{align}

Denote by $G: H \rightarrow \LL_2^0$ the Nemytskii operator associated with $g: \rr \rightarrow \rr$:
\begin{align} \label{df-G}
G(x) g_k(\xi):=g(x(\xi)) g_k(\xi), \quad x \in H,~ k \in \nn,~ \xi \in \OOO.
\end{align}
Our primary condition on the diffusion operator $G$ defined in \eqref{df-G} is the following usual Lipschitz continuity and linear growth conditions.

\begin{ap} \label{ap-g}
There exist nonnegative constants $L_6$, $L_7$, and $L_8$ such that 
\begin{align}  
\|G(x)-G(y)\|_{\LL_2^0}^2 \le L_6 \|x-y\|^2, & \quad x, y \in H, \label{g-lip} \\
\|G(x) \|^2_{\LL_2^0} \le L_7 \|x\|^2+L_8, & \quad x \in H. \label{g-gro} 
\end{align}
Moreover, $G(\dot H^1) \subset \LL_2^1$ and there exist nonnegative constants $L_9$ and $L_{10}$ such that
\begin{align} \label{g-gro-1}
\|G(z) \|^2_{\LL_2^1} \le L_9 \|z\|^2_1+L_{10}, \quad z \in \dot H^1.  
\end{align}
\end{ap}

\begin{rk} 
In the additive noise case $G \equiv {\rm Id}$, the identity operator in $H$, the conditions \eqref{g-lip}--\eqref{g-gro-1} reduce to the assumption $\|(-A)^{1/2} {\bf Q}^{1/2}\|_{HS(H;H)}<\infty$ which was imposed in \cite{QW19} and \cite{CHS21}; see \cite[Remark 2.2]{LQ21} for an example of multiplicative noise case.
\end{rk}

With these preliminaries, \eqref{see-fg} is equivalent to the following infinite-dimensional stochastic evolution equation:
\begin{align} \label{see}
\begin{split}
{\rm d}X_t &=(AX_t+F(X_t)) {\rm d}t+G(X_t) {\rm d}W, \quad t \ge 0;  \\
X(0) & =X_0,
\end{split}
\end{align}
where the random initial datum $X_0$ is assumed to vanish on the boundary $\partial \OOO$ of the physical domain throughout the present paper.

\subsection{Uniform Moments' Estimates of The Exact Solution}

The existence of a unique variational solution of Eq. \eqref{see} in $H$ and $\dot H^1$ was shown in \cite{Cer01} or \cite{KR07} and \cite{Liu13}, respectively, under similar conditions as in Assumptions \ref{ap-f}--\ref{ap-g}.
This solution is also the mild solution of \eqref{see} in \cite{LQ21}, where the author also derived that the solution is indeed continuous a.s. and satisfies certain moments' estimates in a finite time horizon.

Taking advantage of the dissipativity conditions \eqref{F-coe-0} and \eqref{F-coe-1}, we can show the following uniform stability estimate both in $H$ and $\dot H^1$ of Eq. \eqref{see}.
Throughout, we assume without loss of generality that the initial datum $X_0$ of Eq. \eqref{see} is a deterministic function; $X_t^x$ denotes the solution of Eq. \eqref{see} starting from $X_0=x$. 
We will frequently use $C$ to denote generic time-independent constants independent of various discrete parameters that would differ in each appearance.

\begin{prop} \label{prop-sta} 
Let $p \ge 2$, $X_0 \in H$, Assumption \ref{ap-f} and conditions \eqref{g-lip}--\eqref{g-gro} hold.
Assume that $L_2+\frac{p-1}2 L_7<\lambda_1$.
Then there exist positive constants $\gamma_1$ and $C_{\gamma_1}$ such that  
\begin{align}\label{sta-u} 
\ee \|X_t\|^p 
& \le e^{-\gamma_1 t}  \|X_0\|^p + C_{\gamma_1}, \quad t \ge 0.
\end{align}  
Moreover, if $X_0 \in \dot H^1$, $L_2+\frac{p-1}2 L_9 < \lambda_1$, 
and condition \eqref{g-gro-1} hold, then there exist positive constants $\gamma_2$ and $C_{\gamma_2}$ such that 
\begin{align}\label{sta-u-1} 
\ee \|X_t\|^p_1 
& \le e^{-\gamma_2 t}  \|X_0\|^p_1 + C_{\gamma_2}, \quad t \ge 0.
\end{align}  
\end{prop}

\begin{proof}  
Applying It\^o formula (see, e.g., \cite[Theorem 4.2.5]{LR15}) to the functional $\|X\|^p$ leads to
\begin{align} \label{ito-h1}  
{\rm d} \|X\|^p 
& = \Big[ \frac{p(p-2)}2 \|X\|^{p-4}  \sum_{k \in \nn_+} ~ \<X, G(X) {\bf Q}^\frac12 g_k\>^2 \nonumber   \\
& \quad +\frac p2  \|X\|^{p-2} \Big(2 ~ _1\<X, A X+F(X)\>_{-1}
+\|G(X)\|^2_{\LL_2^0} \Big) \Big]  {\rm d}t \nonumber  \\ 
&\quad + p \|X\|^{p-2} \<X, G(X){\rm d}W\>.
\end{align} 
We omit the temporal variable when an integral appears here and after to lighten the notations. 
For $\gamma_1 \in \rr$, it follows from the product rule that 
\begin{align}    \label{eq-exp}
& e^{\gamma_1 t} \|X_t\|^p - \|X_0\|^p \nonumber \\
& = \int_0^t e^{\gamma_1 r} \Big[ \frac{p(p-2)}2 \|X\|^{p-4}  \sum_{k \in \nn_+} \<X, G(X) {\bf Q}^\frac12 g_k\>^2 \nonumber   \\
& \quad + \frac p2 \|X\|^{p-2} (2~_1\<X, A X+F(X)\>_{-1}
+ \|G(X)\|^2_{\LL_2^0} ) 
+ \gamma_1 \|X\|^p \Big]  {\rm d}r \nonumber  \\ 
&\quad + p \int_0^t e^{\gamma_1 r} \|X\|^{p-2} \<X, G(X){\rm d}W\>.
\end{align} 

By the condition \eqref{g-gro} and Young inequality, we have 
\begin{align*}
& \|X\|^{p-4}  \sum_{k \in \nn_+} \<X, G(X) {\bf Q}^\frac12 g_k\>^2 
\le (L_7+\varepsilon) \|X\|^p +C_\varepsilon,   \\
&  \|X\|^{p-2} \|G(X)\|^2_{\LL_2^0} 
\le (L_7+\varepsilon) \|X\|^p +C_\varepsilon, 
\end{align*}
for $X \in H$.   
Combining the above two estimates and \eqref{F-coe-0}, taking expectation on Eq. \eqref{eq-exp}, and noting that $\int_0^t e^{\gamma_1 r} \|X\|^{p-2} \<X, G(X){\rm d}W\>$ is a (real-valued) martingale, we obtain
\begin{align}   \label{est-exp} 
e^{\gamma_1 t} \ee \|X_t\|^p 
& \le \|X_0\|^p - \gamma_1^* \int_0^t e^{\gamma_1 r} \ee \|X\|^p {\rm d}r
+ C_\varepsilon \int_0^t e^{\gamma_1 r}  {\rm d}r,
\end{align}  
with $\gamma_1^*:=p (\lambda_1-L_2-\frac{p-1}2 (L_7+\varepsilon))-\gamma_1$.
As $L_2+\frac{p-1}2 L_7<\lambda_1$, one can take $\varepsilon>0$ and $\gamma_1 >0$ to be  sufficiently small such that $\gamma_1^*>0$. 
Then it follows that 
\begin{align*}    
\ee \|X_t\|^p 
& \le e^{-\gamma_1 t} \|X_0\|^p + C_\varepsilon \int_0^t e^{-\gamma_1 r}  {\rm d}r,
\end{align*}  
from which we obtain \eqref{sta-u}.

Now let $X_0 \in \dot H^1$.
Applying It\^o formula to the functional $e^{\gamma_1 t} \|X_t\|^p_1$ leads to 
\begin{align*}
& e^{\gamma_2 t} \|X_t\|^p_1 -\|X_0\|^p_1 \nonumber  \\ 
& = \int_0^t e^{\gamma_2 r} \Big[ \frac{p(p-2)}2 \|X\|_1^{p-4}  \sum_{k \in \nn_+} \<X, G(X) {\bf Q}^\frac12 g_k\>_1^2 \nonumber   \\
& \quad + \frac p2 \|X\|_1^{p-2} (2 \<X, A X+F(X)\>_1 + \|G(X)\|^2_{\LL_2^1}) 
+ \gamma_2 \|X\|^p_1 \Big]  {\rm d}r \nonumber  \\ 
&\quad + p \int_0^t e^{\gamma_2 r} \|X\|_1^{p-2} \<X, G(X){\rm d}W\>_1.
\end{align*} 
By the conditions \eqref{g-lip} and \eqref{g-gro-1}, we have 
\begin{align*}
& \|X\|_1^{p-4}  \sum_{k \in \nn_+} \<X, G(X) {\bf Q}^\frac12 g_k\>_1^2 
\le (L_9+\varepsilon) \|X\|_1^p +C_\varepsilon, \quad u\in\dot H^1, \\
&  \|X\|_1^{p-2} \|G(X)\|^2_{\LL_2^1} 
\le (L_9+\varepsilon) \|X\|_1^p +C_\varepsilon, \quad u\in\dot H^1.
\end{align*} 
Using techniques similar to those used to establish \eqref{est-exp}, together with the above two estimates, \eqref{F-coe-1}, and the fact that $\int_0^t e^{\gamma_2 r} \|X\|_1^{p-2} \<X, G(X){\rm d}W\>_1$ is a martingale, implies that 
\begin{align*}    
e^{\gamma_2 t} \ee \|X_t\|^p_1 
& \le \|X_0\|^p_1 - \gamma_2^* \int_0^t e^{\gamma_2 r} \ee \|X\|_1^p {\rm d}r
+ C_\varepsilon \int_0^t e^{\gamma_2 r}  {\rm d}r,
\end{align*}  
with $\gamma_2^*:=p (\lambda_1-L_2-\frac{p-1}2 (L_9+\varepsilon))-\gamma_2>0$ for some sufficiently small and positive $\varepsilon$ and $\gamma_2$, as $L_2+\frac{p-1}2 L_9 < \lambda_1$.
Then it follows that 
\begin{align*}    
\ee \|X_t\|_1^p 
& \le e^{-\gamma_2 t} \|X_0\|^p + C_\varepsilon \int_0^t e^{-\gamma_2 (t-r)}  {\rm d}r,
\end{align*}  
from which we conclude \eqref{sta-u-1}. 
\end{proof}

As a by-product of the uniform stability estimate \eqref{sta-u-1}, we have the following uniform temporal H\"older continuity estimate; see \cite[Proposition 5.4]{BIK16} for a similar result in finite-dimensional case.
It will be used in Section \ref{sec4} to derive the time-independent strong error estimates of the proposed implicit numerical schemes.

\begin{cor} 
Let $p \ge 2$, $X_0\in \dot H^1$, and Assumptions \ref{ap-f}--\ref{ap-g} hold. 
Assume that $L_1+\frac{p-1}2 L_9 < \lambda_1$.
Then, there exists a constant $C$ such that 
\begin{align} \label{hol-u} 
\sup_{t >s \ge 0}\frac{\ee \|X_t-X_s\|^p}{(t-s)^{p/2}} 
\le  C (1+\|X_0\|_1^{pq}).
\end{align} 
\end{cor}

\begin{proof} 
For any $t >s \ge 0$, it was shown in \cite[Theorem 3.2]{LQ21} that 
\begin{align*}
& \ee \|X_t-X_s\|^p \\
& \le  C(t-s)^{p/2} (1+\|X_0\|_1^p
+ \sup_{t \ge 0} \ee \|F(X_t)\|^p
+ \sup_{t \ge 0} \ee \|G(X_t)\|_{\LL_2^0}^p).
\end{align*} 
This estimate, together with the conditions \eqref{con-f'}, \eqref{g-gro}, the embeddings \eqref{emb}, and the stability estimate \eqref{sta-u-1}, implies the desired estimate \eqref{hol-u}.
\end{proof}

\subsection{Exponential Ergodicity of The Exact Solution}

It is clear that the solution of Eq. \eqref{see} is a Markov process, see, e.g., \cite[Proposition 4.3.5]{LR15}, so we define the family of Markov semigroup $\{P_t\}_{t \ge 0}$ associated with Eq. \eqref{see} as 
\begin{align*}
(P_t \phi) (x):=\ee \phi(X_t^x), \quad x \in H,~ t \ge 0.
\end{align*} 
 
A probability measure $\mu$ on $H$ is called invariant for $\{P_t\}_{t \in \ii}$ with either $\ii=\rr_+$ or $\ii=\nn$, if 
\begin{align*}
\int_H P_t \phi(x) \mu({\rm d}x)=\mu(\phi):=\int_H \phi(x) \mu({\rm d}x),
\quad \phi \in \CC_b(H), \ t \in \ii.
\end{align*} 
An invariant (probability) measure $\mu$ is called ergodic for $\{P_t\}_{t \ge 0}$ (or Eq. \eqref{see}) or $\{P_m\}_{m \in \nn}$ if for all $\phi \in L^2(H; \mu)$, it holds, respectively, that  
\begin{align*}
\lim_{T \to \infty} \frac1T \int_0^T P_t \phi {\rm d}t\Big(\text{or}~\lim_{m \to \infty} \frac1m \sum_{k=0}^m P_k \phi \Big)=\mu(\phi) 
\quad \text{in}~ L^2(H; \mu).
\end{align*}
 
The following unique ergodicity of Eq. \eqref{see}, with an exponential convergence to the equilibrium, was proved using a remote initial condition method in \cite[Theorem 4.3.9]{LR15}.  

\begin{tm}  \label{tm-erg}
Let Assumption \ref{ap-f} and conditions \eqref{g-lip}-\eqref{g-gro} hold.
Assume that $(L_1+L_6/2) \vee (L_2+L_7/2)<\lambda_1$.
Then a unique invariant measure $\pi$  in $H$ exists for Eq. \eqref{see}.
Moreover, for any $t \ge 0$, $x \in H$, and $\phi \in {\rm Lip}_b(H)$, there exist a positive constant $\gamma_3$ such that  
\begin{align} \label{mix}
|P_t \phi(x)-\pi(\phi)| \le  \exp (-\gamma_3 t/2)|\phi|_{\rm Lip} \int_H \|x-y\| \pi ({\rm d} y).
\end{align}   
\end{tm}

\begin{rk}
A Markov semigroup $\{P_t\}_{t \in \ii}$ with $\ii=\rr_+$ or $\ii=\nn$ satisfying an exponential convergence to the equilibrium, like \eqref{mix}, is called exponentially mixing.
It is clear that once $\{P_t\}$ is exponentially mixing for the invariant measure $\mu$, then the following strong law of large numbers (implying the ergodicity) holds:
\begin{align*} 
\lim_{T \to \infty} \frac1T \int_0^T P_t \phi(x) {\rm d}t~
\Big(\text{or}~\lim_{m \to \infty} \frac1m \sum_{k=0}^m P_k \phi(x)\Big)
=\mu(\phi),
\quad \forall~ x \in H, ~\mu\text{-a.s.}
\end{align*}
\end{rk}

\section{Exponential Ergodicity and Uniform Estimates of Galerkin-based Euler Scheme}
\label{sec3}

This section will first establish the exponential ergodicity of the drift-implicit Euler (DIE) semi-discrete and Galerkin-based fully discrete schemes. 
Then, we will derive uniform, strong error estimates of this fully discrete scheme. 
At the end of this section, some remarks on the generalization to the Milstein-type scheme and its Galerkin discretization are given.

\subsection{Exponential Ergodicity of DIE Scheme}

Let $\tau \in (0, 1)$. For an $m \in \nn$, define $t_m:=m \tau$.
We discrete Eq. \eqref{see} in time with a DIE scheme, and the resulting temporal semi-discrete problem is to find an $\dot H^1$-valued Markov chain $\{X_m: \ m \in \nn_+\}$ such that
\begin{align}\label{die} \tag{DIE} 
&X_{m+1}
=X_m+\tau A X_{m+1}
+\tau F(X_{m+1})
+ G(X_m) \delta_m W, 
\end{align}
starting from $X_0 \in H$, where $\delta_m W=W(t_{m+1})-W(t_m)$, $m \in \nn$.
This DIE scheme had been widely studied; see, e.g., \cite{CHS21}, \cite{FLZ17}, \cite{LQ21}, \cite{QW19}, and references therein.
Throughout this part we assume that $\tau \in (0, 1)$ with $(L_1-\frac{\lambda_1}{\lambda_1+1}) \tau<1$ so that \eqref{die} is well-posedness; see \cite[Lemma 4.1]{LL24}.

We begin with the following uniform moments' stability of the above DIE scheme \eqref{die}.
It should be pointed out that there is no further restriction on the time step size $\tau \in (0, 1)$; see \cite[Theorems 4.1 and 4.2]{LMW23} for an upper bound requirement on $\tau$ for the numerical ergodicity of monotone SODEs.

\begin{lm} \label{lm-sta-uk} 
Let $X_0 \in H$, Assumption \ref{ap-f} and conditions \eqref{g-lip}-\eqref{g-gro} hold, and $\{X_m: m \in \nn\}$ be the solution of the DIE scheme \eqref{die} starting from $X_0$.
Assume that $L_2+L_7/2<\lambda_1$.
Then there exist positive constants $\gamma_5$, $C_{\gamma_5}$, and $C$ such that for any $\tau \in (0, 1)$, 
\begin{align} 
& \ee \|X_m\|^2 \le e^{-\gamma_5 t_m} \|X_0\|^2+C_{\gamma_5}, \quad m \in \nn, \label{sta-uk} \\
& \sup_{m \in \nn_+} \frac1m \sum_{k=0}^m \ee \|X_k\|_1^2 \le C( 1 + \|X_0\|^2).  \label{sta-uk+}
\end{align}
\end{lm} 
 
\begin{proof} 
For simplicity, set $F_k=F(X_k)$ and $G_k=G(X_k)$ for $k \in \nn$.
Testing \eqref{die} using $\<\cdot, \cdot\>$-inner product with $X_{k+1}$ and using the elementary equality 
\begin{align} \label{ab}
2 \<x-y, x\> =\|x \|^2- \|y\|^2+\|x-y\|^2,
\quad x, y \in H,
\end{align}  
we have  
\begin{align} \label{sta-uk0}
& \|X_{k+1}\|^2 - \|X_k\|^2 +  \|X_{k+1}-X_k\|^2 \nonumber \\
& = 2  ~ _1\<X_{k+1}, \Delta X_{k+1}+ F_{k+1}\>_{-1} \tau\nonumber \\
& \quad + 2 \<X_{k+1}-X_k, G_k \delta_k W\>
+ 2 \<X_k, G_k \delta_k W\>.
\end{align}   
By the estimate \eqref{F-coe-0} and Cauchy--Schwarz inequality, we have
\begin{align*} 
&  \|X_{k+1}\|^2 - \|X_k\|^2  \\
& \le - 2 (\lambda_1-L_2) \|X_{k+1}\|^2 \tau 
+ \|G_k \delta_k W\|^2 + 2\<X_k, G_k \delta_k W\>+ 2 L_3 \tau. 
\end{align*}  
It follows that
\begin{align*}  
& (1+2 (\lambda_1-L_2) \tau) \|X_{k+1}\|^2 
\le \|X_k\|^2+\|G_k \delta_k W\|^2 
+ 2\<X_k, G_k \delta_k W\>+2 L_3 \tau.
\end{align*}  

Taking expectations on both sides, noting the fact that both $G_k$ and $X_k$ are independent of  $\delta_k W$, using It\^o isometry and \eqref{g-gro}, we get
\begin{align*}  
& (1+2 (\lambda_1-L_2) \tau) \ee \|X_{k+1}\|^2 
\le (1+L_7 \tau)  \ee \|X_k\|^2 + (2 L_3+L_8) \tau,
\end{align*} 
from which we obtain
\begin{align*}  
\ee \|X_k\|^2 
& \le \Big(\frac{1+L_7 \tau}{1+2 (\lambda_1-L_2) \tau}\Big)^k \|X_0\|^2 \\
& \quad + \frac{(2 L_3+L_8) \tau}{1+2 (\lambda_1-L_2) \tau} \sum_{j=0}^{k-1} \Big(\frac{1+L_7 \tau}{1+2 (\lambda_1-L_2) \tau}\Big)^j, 
\quad k \in \nn_+.
\end{align*}   
Note that 
\begin{align} \label{e}
a^k < e^{-(1-a) k}, \quad \forall~ a \in (0, 1), 
\end{align}  
and that $\frac{1+L_7 \tau}{1+2 (\lambda_1-L_2) \tau}<1$ for any $\tau \in (0, 1)$, which is ensured by the condition $L_2+L_7/2<\lambda_1$, we conclude \eqref{sta-uk} with 
$\gamma_5:=\frac{2 (\lambda_1-L_2)-L_7}{1+2 (\lambda_1-L_2) \tau}$ and 
$C_{\gamma_5}:=\frac{2 L_3+L_8}{2 (\lambda_1-L_2)-L_7}$.

Subtracting $2 \varepsilon  ~ _1\<X_{k+1}, \Delta X_{k+1}\>_{-1} \tau$, with small $\varepsilon>0$, from the first term on the right-hand side in Eq. \eqref{sta-uk0}, and using the arguments in the above, we obtain 
\begin{align*}  
& (1+2 (\lambda_1-L_2-\varepsilon) \tau) \ee \|X_{k+1}\|^2 
+ 2 \varepsilon \ee \|\nabla X_{k+1}\|^2 \tau \\
& \le (1+L_7 \tau)  \ee \|X_k\|^2 + (2 L_3+L_8) \tau,
\end{align*} 
from which we derive 
\begin{align*}  
& \ee \|X_k\|^2 + \frac{2 \varepsilon}{1+2 (\lambda_1-L_2-\varepsilon) \tau} \sum_{j=0}^k \ee \|\nabla X_j\|^2 \tau \\
& \le \Big(\frac{1+L_7 \tau}{1+2 (\lambda_1-L_2-\varepsilon) \tau}\Big)^k \|X_0\|^2 \\
& \quad + \frac{(2 L_3+L_8) \tau}{1+2 (\lambda_1-L_2) \tau} \sum_{j=0}^{k-1} \Big(\frac{1+L_7 \tau}{1+2 (\lambda_1-L_2) \tau}\Big)^j, 
\quad k \in \nn_+.
\end{align*}   
Then we conclude \eqref{sta-uk+} by the above estimate, the elementary estimate \eqref{e}, and \eqref{sta-uk}. 
\end{proof}

In the discrete case, there is no It\^o formula anymore. 
Whether $\{e^{\gamma t_m} \|X_m^x-X_m^y\|^2: m \in \nn\}$, for some $\gamma$, is discrete supermartingale is unknown, where $\{X_m^x: m \in \nn\}$ and $\{X_m^y: m \in \nn\}$ are denoted by the solution of the DIE scheme \eqref{die} starting from $x$ and $y$, respectively.
Fortunately, using the monotonicity conditions in Assumptions \ref{ap-f}-\ref{ap-g}, one can show the following discrete uniform continuity dependence estimate similar to the finite-dimensional case in \cite[Theorem 7.6]{LMY19} or \cite[Lemma 4.2]{LMW23}.

\begin{lm} \label{lm-con-uk}  
Let $x, y \in H$, Assumption \ref{ap-f} and conditions \eqref{g-lip}-\eqref{g-gro} hold.
Assume that $L_1+L_6/2<\lambda_1$. 
Then there exist a positive constant $\gamma_6$ such that for any $\tau \in (0, 1)$,  
\begin{align} \label{con-uk}
\ee \|X_m^x-X_m^y\|^2 \le e^{-\gamma_6 t_m} \|x-y\|^2,
\quad m \in \nn.
\end{align} 
Consequently, $\{P_m\}_{m \in \nn_+}$ is Feller, i.e., if $\phi \in \CC_b(H)$, then $P_m \phi \in \CC_b(H)$ for all $m \in \nn_+$ (or equivalently, $P_1 \phi \in \CC_b(H)$).  
\end{lm}

\begin{proof} 
For $k \in \nn$, set $F_k^x=F(X_k^x)$, $F_k^y=F(X_k^y)$, $G_k^x=G(X_k^x)$, $G_k^y=G(X_k^y)$, and $E_k=X_k^x-X_k^y$. 
Using the same technique in the proof of Lemma \ref{lm-sta-uk} to derive \eqref{sta-uk0}, we have  
\begin{align*} 
& \|E_{k+1}\|^2 - \|E_k\|^2 +  \|E_{k+1}-E_k\|^2  \\
& = 2  ~ _1\<E_{k+1}, \Delta E_{k+1}+ F^x_{k+1}-F^y_{k+1}\>_{-1} \tau \\
& \quad + 2 \<E_{k+1}-E_k, (G^x_k-G^y_k) \delta_k W\>
+ 2 \<E_k, (G^x_k-G^y_k) \delta_k W\>.
\end{align*}   
It follows from the estimate \eqref{F-mon-0} and Cauchy--Schwarz inequality that 
\begin{align*} 
&  \|E_{k+1}\|^2 - \|E_k\|^2  \\ 
&\le - 2 (\lambda_1-L_1) \|E_{k+1}\|^2 \tau 
+ \|(G^x_k-G^y_k) \delta_k W\|^2 \\
& \quad + 2\<E_k, (G^x_k-G^y_k) \delta_k W\>. 
\end{align*}   
Taking expectations on both sides, noting the fact that both $G^x_k-G^y_k$ and $X_k$ are independent of  $\delta_k W$, using It\^o isometry and \eqref{g-lip}, we obtain
\begin{align*}  
& (1+2 (\lambda_1-L_1) \tau) \ee \|E_{k+1}\|^2
\le (1+L_6 \tau) \ee \|E_k\|^2,
\end{align*}   
from which we obtain
\begin{align*}  
\ee \|E_k\|^2 
\le \Big(\frac{1+L_6 \tau}{1+2 (\lambda_1-L_1) \tau}\Big)^k \|x-y\|^2, 
\quad k \in \nn.
\end{align*}    
Then we conclude \eqref{con-uk} with 
$\gamma_6:=\frac{2 (\lambda_1-L_1)-L_6}{2 (\lambda_1-L_1)+1}$ by the above estimate and \eqref{e}.
\end{proof}

We can show that the DIE scheme \eqref{die} is ergodic with a unique invariant measure.
Moreover, it converges exponentially to the equilibrium.

\begin{tm} \label{tm-erg-uk} 
Let Assumption \ref{ap-f} and conditions \eqref{g-lip}-\eqref{g-gro} hold.
Assume that $(L_1+L_6/2) \vee (L_2+L_7/2)<\lambda_1$.
Then for any $\tau \in (0, 1)$, the DIE scheme \eqref{die} possesses a unique invariant measure $\pi_\tau$ in $H$ which is exponentially mixing: 
\begin{align} \label{mix-uk}
|P_m \phi(x)-\pi_\tau(\phi)| \le \exp(-\gamma_6 t_m/2)|\phi|_{\rm Lip} \int_H \|x-y\| \pi_\tau ({\rm d} y),
\end{align}  
for any $m \in \nn$, $x \in H$, and $\phi \in {\rm Lip}_b(H)$, where $\gamma_6$ appeared at \eqref{con-uk}.
\end{tm}

\begin{proof}
Define a sequence $\{\mu^\tau_m\}_{m \in \nn_+}$ of probability measures by \begin{align*}
\mu^\tau_m(B):= \frac1m \sum_{k=0}^m P_k \chi_B(0) dt, \quad B \in \BB(H), ~m \in \nn_+. 
\end{align*}   
By Chebyshev inequality and the temporal average estimate \eqref{sta-uk+} (with $X_0=0$), 
for any $\varepsilon>0$, there exists a positive scalar $I_\varepsilon$ and a compact set $B=\{x \in H:~ \|x\|_1 \le I_\varepsilon\} \subset \BB(H)$ such that 
\begin{align*}  
\mu^\tau_m(B^c)
& =\frac1m \sum_{k=0}^m \pp(\|X_k^0\|_1>I_\varepsilon) dt  
\le \frac{I^{-2}_\varepsilon} m \sum_{k=0}^m \ee \|X_k^0\|_1^2 
\le C I^{-2}_\varepsilon<\varepsilon,
\end{align*}
from which we have that $\{\mu^\tau_m\}_{m \in \nn_+}$ is tight.
Then we conclude the existence of an invariant measure $\pi_\tau$ of $\{P_m\}_{m \in \nn}$ with $\pi_\tau(\|\cdot\|^2)<\infty$, via Krylov--Bogoliubov theorem (see, e.g., \cite[Theorem 1.2]{HW19}), by the tightness of $\{\mu^\tau_m\}_{m \in \nn_+}$ and the Feller property of $\{P_m\}_{m \in \nn}$ in Lemma \ref{lm-con-uk}.

Now let $\mu$ be any invariant measure of $\{P_m\}_{m \ge 0}$.
By the property of invariant measure and continuous dependence estimate \eqref{con-uk}, 
\begin{align*}
|P_m \phi(x)-\mu(\phi)| 
& =\Big| \int_H (P_m \phi(x)-P_m \phi(y)) \mu ({\rm d}y) \Big| \\
& \le  |\phi|_{\rm Lip} \int_H \ee \|X_t^x-X_t^y\| \mu ({\rm d}y) \\
& \le  \exp (-\gamma_6 t_m/2)|\phi|_{\rm Lip} \int_H \|x-y\| \mu({\rm d}y),
\end{align*} 
for any $x \in H$ and $m \ge 0$. 
Taking $m \to \infty$, we have that all such invariant measures coincide, i.e., $\mu=\pi_\tau$, and thus \eqref{mix-uk} holds.    
\end{proof}

\subsection{Exponential Ergodicity of DIEG Scheme}

Next, let us generalize the arguments in the previous two parts to spatial discretizations via Galerkin approximations of the temporal DIE scheme \eqref{die}.
Then, we derive the time-independent strong convergence rate for the fully discrete scheme towards the exact solution of Eq. \eqref{see}.

Let $h\in (0,1)$, $\TT_h$ be a regular family of partitions of $\OOO$ with maximal length $h$, and $V_h \subset \dot H^1$ be the space of continuous functions on $\bar \OOO$ which are piecewise linear over $\TT_h$ and vanish on the boundary $\partial \OOO$.
Let $A_h: V_h \rightarrow V_h$ and $\PP_h: \dot H^{-1} \rightarrow V_h$ be the discrete Laplacian and generalized orthogonal projection operators, respectively, defined by 
\begin{align*}  
\<v^h, A_h x^h\> & =-\<\nabla x^h, \nabla v^h\>,
\quad x^h, v^h\in V_h,  \\
\<v^h, \PP_h z\> & =_1\<v^h, z\>_{-1},
\quad z \in \dot H^{-1},\ v^h\in V_h.   
\end{align*}

We discrete DIE \eqref{die} in space with the Galerkin method, called the DIEG  scheme, and the resulting fully discrete problem is to find a $V^h$-valued discrete process $\{X^h_m:\ m \in \nn\}$ such that
\begin{align}\label{die-g} \tag{DIEG}
\begin{split}
&X^h_{m+1}
=X^h_m+\tau A_h X^h_{m+1}
+\tau \PP_h F(X^h_{m+1})
+\PP_h G(X^h_m) \delta_m W, \\
& X^h_0=\PP_h X_0,
\end{split}
\end{align} 
$m \in \nn$, which had been widely studied; see, e.g., \cite{FLZ17}, \cite{QW19}, \cite{LQ21}.
It is clear that the DIEG scheme \eqref{die-g} is equivalent to the scheme
\begin{align*}
X^h_{m+1}=S_{h,\tau} X^h_m+\tau S_{h,\tau} \PP_h F(X^h_{m+1})
+S_{h,\tau} \PP_h G(X^h_m) \delta_m W,
\quad m\in \nn,
\end{align*}
with initial datum $X^h_0=\PP_h X_0$, where $S_{h,\tau}:=({\rm Id}-\tau A_h)^{-1}$ is a space-time approximation of the continuous semigroup $S$ in one step.
Iterating the above equation for $m$-times, we obtain 
\begin{align*}
X^h_{m+1}
& =S_{h,\tau}^{m+1} X^h_0+\tau \sum_{i=0}^m S_{h,\tau}^{m+1-i} \PP_h F(X^h_{i+1})   \\
& \quad +\sum_{i=0}^m S_{h,\tau}^{m+1-i} \PP_h G(X^h_i) \delta_i W,
\quad m \in \nn.
\end{align*}
Throughout this part we assume that $\tau \in (0, 1)$ with $(L_1-\lambda_1) \tau<1$ so that \eqref{die-g} is well-posedness.

As in Lemmas \ref{lm-sta-uk} and \ref{lm-con-uk}, we can show the uniform stability and continuous dependence on initial data independent of $h$, where the resulting constants appeared at Lemmas \ref{lm-sta-uk} and \ref{lm-con-uk}.
The proof is analogous to that of Lemmas \ref{lm-sta-uk} and \ref{lm-con-uk}, noting that $\PP_h$ is an orthogonal projection operator also contracted, so we omit the details.

\begin{lm}  
Let $X_0, x, y \in H$, Assumption \ref{ap-f} and conditions \eqref{g-lip}-\eqref{g-gro} hold.
Denote by $\{X^h_m: m \in \nn\}$, $\{X_m^{h,x}: m \in \nn\}$, and $\{X_m^{h,y}: m \in \nn\}$ the solutions of the DIEG scheme \eqref{die-g} starting from $X_0$, $x$, and $y$, respectively.
\begin{itemize}
\item[i)]
Assume that $L_2+L_7/2<\lambda_1$.
Then for any $\tau \in (0, 1)$, 
\begin{align} \label{sta-ukh}
\ee \|X_m^h\|^2 \le e^{-\gamma_5 t_m} \|X_0\|^2+C_{\gamma_5}, \quad m \in \nn.
\end{align}  

\item[ii)]
Assume that $L_1+L_6/2<\lambda_1$. 
Then for any $\tau \in (0, 1)$,
\begin{align} \label{con-ukh}
\ee \|X_m^{h,x}-X_m^{h,y}\|^2 \le e^{-\gamma_6 t_m} \|X-Y\|^2.
\quad m \in \nn.
\end{align}  
Consequently, the corresponding fully discrete Markov semigroup, denoted by $\{P_m^h\}_{m \in \nn}$, of the DIEG scheme \eqref{die-g} is Feller.  
\end{itemize}
\end{lm}

We can show that the DIEG scheme \eqref{die-g} is exponentially mixing.
On the contrary to the case of DIE scheme \eqref{die}, the dimension of state space $V_h$ of the DIEG scheme \eqref{die-g} is finite.

\begin{tm} \label{tm-erg-ukh} 
Let Assumption \ref{ap-f} and conditions \eqref{g-lip}-\eqref{g-gro} hold.
Assume that $(L_1+L_6/2) \vee (L_2+L_7/2)<\lambda_1$.
Then for any $\tau \in (0, 1)$ and $h \in (0, 1)$, the DIEG scheme \eqref{die-g} possesses a unique invariant measure $\pi_\tau^h$  in $V_h$ which is exponentially mixing: 
\begin{align} \label{mix-ukh}
|P_m^h \phi(x)-\pi_\tau^h (\phi)| \le \exp(-\gamma_6 t_m/2)|\phi|_{\rm Lip}\int_H \|x-y\| \pi_\tau^h ({\rm d} y),
\end{align}      
for any $m \in \nn$, $x \in H$, and $\phi \in {\rm Lip}_b(H)$, where $\gamma_6$ appeared at \eqref{con-uk}.
\end{tm}

\begin{proof}
Consider the sequence $\{\mu^\tau_m\}_{m \in \nn_+}$ of probability measures
\begin{align*}
\mu^{\tau,h}_m(B):= \frac1m \sum_{k=0}^m P_k^h \chi_B(0) dt, \quad B \in \BB(V_h), ~m \in \nn_+. 
\end{align*}    
By Chebyshev inequality and the uniform estimate \eqref{sta-ukh}, 
for any $\varepsilon>0$, there exists a positive scalar $J_\varepsilon$ and a compact set $B=\{x \in V_h:~ \|x\| \le J_\varepsilon\} \subset \BB(H)$ such that 
\begin{align*}  
\mu^{\tau,h}_m(B^c)
& =\frac1m \sum_{k=0}^m \pp(\|X_k^{h, 0}\|>J_\varepsilon) dt  \\
& \le J^{-2}_\varepsilon \sup_{k \in \nn} \ee \|X_k^{h, 0}\|^2 
\le C J^{-2}_\varepsilon<\varepsilon,
\end{align*}
from which we have that $\{\mu_t\}_{T \ge 1}$ is tight. 
The exponential mixing estimate \eqref{mix-ukh}  and the existence of a unique invariant measure follow from the continuous dependence estimate \eqref{con-ukh} and similar arguments in the proof of Theorem \ref{tm-erg-uk}.
\end{proof}

\subsection{Uniform Strong Error estimate of DIEG Scheme}

In this part, we aim to derive time-independent strong error estimates of the proposed DIEG scheme \eqref{die-g} towards the exact solution of Eq. \eqref{see}.

We will combine semigroup and variational theories as in \cite{LQ21}.
To this end, for $m\in \nn$, let us denote by $E_{h,\tau}(t)=S(t)-S_{h,\tau}^{m+1} \PP_h$ for $t\in (t_m, t_{m+1}]$, and define the auxiliary process
(see \cite{LQ21} and \cite{QW19}) 
\begin{align}\label{aux1}
\widetilde{X}^h_{m+1}
& =S_{h,\tau}^{m+1} X^h_0
+\tau \sum_{i=0}^m S_{h,\tau}^{m+1-i} \PP_h F(X(t_{i+1})) \nonumber \\
& \quad +\sum_{i=0}^m S_{h,\tau}^{m+1-i} \PP_h G(X(t_i)) \delta_i W.
\end{align}  
In view of the well-known uniform boundedness of $S_{h,\tau}$, the conditions \eqref{con-f'} and \eqref{g-gro-1}, and the uniform stability \eqref{sta-u-1}, it is not difficult to show that for any $p \ge 2$, 
\begin{align}\label{sta-aux}
\sup_{k \in \nn} \ee \|\widetilde{X}^h_k\|_1^p
\le C (1+ \sup_{t \ge 0} \ee \|X_t\|_1^{p q})
\le C (1+ \|X_0\|_1^{p q}).
\end{align}  

We begin with the following uniform strong error estimate between the exact solution $u$ of Eq. \eqref{see} and the auxiliary process $\{\widetilde{X}^h_m\}_{m \in \nn}$ defined by \eqref{aux1}.
It is proved using the idea in \cite[Lemma 4.2]{LQ21} and the uniform moments' bound \eqref{sta-u-1}.

\begin{lm}  
Let $X_0\in \dot H^1$ and Assumptions \ref{ap-f}-\ref{ap-g} hold.
Assume that $p\ge 2$ and $L_1+\frac{p(2q-1)-1}2 L_9 < \lambda_1$.
Then there exist a positive constant $C$ such that for any $\tau \in (0, 1)$, 
\begin{align} \label{u-uhathm}
\sup_{m \in \nn} \ee \|X(t_m)-\widetilde{X}^h_m\|^p
\le C (1+ \|X_0\|_1^{p(2q-1)}) (h^p+\tau^{p/2}).
\end{align} 
\end{lm}

\begin{proof}
Let $m\in \nn$.
In  \cite[Lemma 4.2]{LQ21}, the error 
\begin{align*}  
J^{m+1}:= (\ee \|X(t_{m+1})-\widetilde{X}^h_{m+1}\|^p)^{1/p}
\end{align*} 
is divided into three parts $J^{m+1}_i$, $i=1,2,3$, respectively, which satisfies
\begin{align*}  
J^{m+1}_1 & =(\ee \| E_{h,\tau}(t_{m+1}) X_0\|^p)^{1/p}
\le C (h+\tau^{1/2}) \|X_0\|_1,   \\
J^{m+1}_2 & = \Big(\ee \Big\|\int_0^{t_{m+1}} S_{t_{m+1}-r} F(X_r) {\rm d}r 
-\tau \sum_{i=0}^m S_{h,\tau}^{m+1-i} \PP_h F(X(t_{i+1})) \Big\|^p \Big)^{1/p}   \\
&\le C \Big(1+ \sup_{t \ge 0} (\ee \|X_t\|_1^{2p(q-1)})^\frac1{2p} \Big)
 \Big( \sup_{t \neq s} \frac{\ee \|X_t-X_s\|^{2p}}{|t-s|^p} \Big)^\frac1{2p} \tau^\frac12 \\
& \quad + C (h+\tau^{1/2})  \sup_{t \ge 0} (\ee \|F(X_t)\|^p)^{1/p},   \\
J^{m+1}_3 & = \Big(\ee \Big\|\int_0^{t_{m+1}} S_{t_{m+1}-r} G(X_r) {\rm d}W_r 
-\sum_{i=0}^m S_{h,\tau}^{m+1-i} \PP_h G(X(t_i)) \delta_i W \Big\|^p\Big)^{1/p} \\
&\le C \Big( \sup_{t \neq s} \frac{\ee \|X_t-X_s\|^{2p}}{|t-s|^p} \Big)^\frac1{2p} \tau^\frac12
+C (h+\tau^\frac 12) (1+\sup_{t \ge 0} (\ee \|X_t\|_1^p)^\frac1p).
\end{align*} 
Putting the above three estimates together results in
\begin{align*} 
J^{m+1}
& \le C \Big(1+\|X_0\|_1+\Big(1+\sup_{t \ge 0} (\ee \|X_t\|_1^{2p(q-1)})^\frac1{2p} \Big) \\
& \qquad \times \Big( \sup_{t \neq s} \frac{\ee \|X_t-X_s\|^{2p}}{|t-s|^p} \Big)^\frac1{2p}  \\
& \quad +  \sup_{t \ge 0} (\ee \|F(X_t)\|^p)^\frac1p
+ \sup_{t \ge 0} (\ee \|X_t\|_1^p)^\frac1p \Big) (h+\tau^\frac12).
\end{align*}
This inequality, in combination with the estimate on $J^0$ that  
\begin{align*} 
J^0=\|X_0-\PP_h X_0\|_{L_\omega^p L^2}
\le C h \|X_0\|_ 1, 
\end{align*}
the condition \eqref{con-f'}, the embeddings \eqref{emb}, the uniform stability estimate \eqref{sta-u-1}, and the uniform H\"older estimate \eqref{hol-u}, completes the proof of \eqref{u-uhathm}. 
\end{proof}

Now, we are at the position of deriving the uniform strong error estimate between the exact solution $u$ of Eq. \eqref{see} and the DIEG scheme \eqref{die-g}. 
  
\begin{tm}  \label{tm-err}
Let $X_0 \in \dot H^1$, Assumptions \ref{ap-f}-\ref{ap-g} hold.
Assume that $L_1+\frac{2(q^2+q-1)-1}2 L_9 < \lambda_1$.
Then for any $\tau \in (0, 1)$ there exist a positive constant $C$ such that    
\begin{align} \label{err} 
& \sup_{k \in \nn} \ee \|X_{t_k}-X^h_k\|^2 
\le C (1+ \|X_0\|_1^{2(q^2+q-1)}) (h^2+\tau).
\end{align}  
\end{tm}

\begin{proof}
For $k \in \nn$, denote  
$e^h_k: =\widetilde{X}^h_k-X^h_k$.
From \eqref{aux1}, it is clear that 
\begin{align*}
\widetilde{X}^h_{m+1}
=\widetilde{X}_m^h+\tau A_h \widetilde{X}^h_{m+1}
+\tau \PP_h F(X_{t_{m+1}})
+\PP_h G(X_{t_m}) \delta_m W.
\end{align*} 
Then, by \eqref{die-g} and the above equality, we have 
\begin{align*}
e^h_{k+1}-e^h_k 
& = [\Delta_h e^h_{k+1}+ \PP_h (F(X_{t_{k+1}})-F(X^h_{k+1}) ] \tau   \\
& \quad + \PP_h (G(X_{t_k})-G(X^h_k)) \delta_k W.
\end{align*} 
Testing with $e^h_{k+1}$ and using the inequality \eqref{ab}, the condition \eqref{F-mon-0}, Cauchy--Schwarz inequality, and the elementary inequality $\|x+y\|^2 \le (1+\varepsilon) \|x\|^2 +(1+\varepsilon^{-1}) \|y\|^2$ for $x,y \in H$, we obtain
\begin{align*} 
& \|e^h_{k+1}\|^2 - \|e^h_k\|^2 
+ \|e^h_{k+1}-e^h_k\|^2   \\
&=2 ~_1\<e^h_{k+1},  F(X_{t_{k+1}})-F(\widetilde{X}^h_{k+1})\>_{-1} \tau \\
& \quad + 2 ~_1\<e^h_{k+1}, \Delta e^h_{k+1}+ F(\widetilde{X}^h_{k+1})-F(X^h_{k+1}) \>_{-1} \tau \\
& \quad + 2 \<e^h_{k+1}-e^h_k, (G(X_{t_k})-G(X^h_k)) \delta_k W\> \\
& \quad + 2 \<e^h_k, (G(X_{t_k})-G(X^h_k)) \delta_k W\>\\
&\le  -2 (\lambda_1-L_1-\varepsilon) \tau \|e^h_{k+1}\|^2  
+ \|e^h_{k+1}-e^h_k\|^2  \\
&\quad 
+ 2 \<e^h_k, (G(X_{t_k})-G(X^h_k)) \delta_k W\> +C_\varepsilon \tau \|F(X_{t_{k+1}})-F(\widetilde{X}^h_{k+1})\|^2_{-1}  \nonumber \\
&\quad + (1+\varepsilon^{-1}) \|(G(X_{t_k})-G(\widetilde{X}^h_k)) \delta_k W\|^2 \\
& \quad +  (1+\varepsilon) \|(G(\widetilde{X}^h_k)-G(X^h_k)) \delta_k W\|^2.
\end{align*} 
It follows that 
\begin{align*}  
& (1-2 (\lambda_1-L_1-\varepsilon) \tau) \|e^h_{k+1}\|^2    \\
&\le \|e^h_k\|^2 + 2 \<e^h_k, (G(X_{t_k})-G(X^h_k)) \delta_k W\> \\
& \quad + (1+\varepsilon^{-1}) \|(G(X_{t_k})-G(\widetilde{X}^h_k)) \delta_k W\|^2  \\
&\quad +C_\varepsilon \tau \|F(X_{t_{k+1}})-F(\widetilde{X}^h_{k+1})\|^2_{-1}  \\
& \quad + (1+\varepsilon) \|(G(\widetilde{X}^h_k)-G(X^h_k)) \delta_k W\|^2.
\end{align*}
Now, taking expectation on both sides of the above inequality and using the fact 
$\ee \<e^h_k, (G(X_{t_k})-G(X^h_k)) \delta_k W\>=0$
and the estimates  
\begin{align*} 
\sup_{k \in \nn}  \ee \|F(X_{t_{k+1}})-F(\widetilde{X}^h_{k+1})\|^2_{-1} 
& \le C (1+ \|X_0\|_1^{2(q^2+q-1)})(h^2+\tau), \\
\sup_{k \in \nn} \ee \|G(X_{t_k})-G(\widetilde{X}^h_k)\|_{\LL_2^0}^2
& \le C (1+ \|X_0\|_1^{2(2q-1)})(h^2+\tau), 
\end{align*}  
which follow from \eqref{u-uhathm} and the conditions \eqref{con-f'} and \eqref{g-lip} (see \cite[Theorem 4.1]{Liu22} for details), the embeddings \eqref{emb}, and the uniform stability estimates \eqref{sta-u} and \eqref{sta-aux}, we get
\begin{align*}  
& (1+2 (\lambda_1-L_1-\varepsilon) \tau) \ee \|e^h_{k+1}\|^2 \\
&  \le (1+ L_6(1+\varepsilon) \tau) \ee \|e^h_k\|^2 
+ C_\varepsilon (1+ \|X_0\|_1^{2(2q-1)})(h^2+\tau)  \tau,
\end{align*}  
from which we obtain
\begin{align*}  
\ee \|e^h_k\|^2 
& \le \Big(\frac{1+ L_6(1+\varepsilon) \tau}{1+2 (\lambda_1-L_1-\varepsilon) \tau}\Big)^k \|e^h_0\|^2 \\
& \quad + \frac{C_\varepsilon (1+ \|X_0\|_1^{2(2q-1)}) (h^2+\tau)  \tau}{1+2 (\lambda_1-L_1-\varepsilon) \tau} \sum_{j=0}^{k-1} \Big(\frac{1+ L_6(1+\varepsilon) \tau}{1+2 (\lambda_1-L_1-\varepsilon) \tau}\Big)^j,
\end{align*}   
for $k \in \nn$.
Now we can conclude \eqref{err} by the above estimate, the inequality \eqref{e}, the estimate \eqref{u-uhathm}, and the estimate $ \ee \|X_{t_k}-X^h_k\|^2  \le 2  \ee \|X_{t_k}-\widetilde{X}^h_k\|^2 + 2 \ee \|e^h_k\|^2$.
\end{proof}

Finally, we can apply the previous arguments to the following DIE Milstein (DIEM) scheme and its Galerkin full discretization under an additional infinite-dimensional commutativity condition, see \cite[Assumption 5.2 and Example 5.1]{LQ21}, which is frequently used to construct temporal high order scheme:
\begin{align}  
Y_{m+1}
&=Y_m+\tau A Y_{m+1}
+\tau F(X_{m+1})
+ G(X_m) \delta_m W  \nonumber \\
& \quad + DG(Y_m) G(Y_m) 
\Big[\int_{t_m}^{t_{m+1}} (W(r)-W(t_m)) {\rm d}W_r \Big],  
\quad Y_0:=X_0. \label{die-m}  \tag{DIEM}
\end{align} 
The proofs are analogous to the previous arguments, so we omit the details. 

\begin{rk}
\begin{enumerate}
\item 
Similar to the DIE scheme \eqref{die}, we have the following uniform moments' stability and continuity dependence estimates for \eqref{die-m}: 
there exist positive constants $\gamma_7$, $C_{\gamma_7}$, $C$, $\gamma_8$ and $\tau_0$ such that for any $X_0, x, y \in H$, $m \in \nn$, and $\tau \in (0, \tau_0)$, 
\begin{align*} 
& \ee \|Y_m\|^2 \le e^{-\gamma_7 t_m} \|X_0\|^2+C_{\gamma_7}, \\ 
& \sup_{n \in \nn_+} \frac1n \sum_{k=0}^n \ee \|Y_k\|_1^2 \le C( 1 + \|X_0\|^2), \\
& \ee \|Y_m^x-Y_m^y\|^2 \le e^{-\gamma_8 t_m} \|x-y\|^2.
\end{align*}  
In this case, there is an upper bound restriction on the time step size $\tau$.
We note that such restriction also appeared in \cite{WL19} where the authors derived numerical ergodicity of the Milstein scheme for Lipschitz SODEs, as one needs to control the last term on the right-hand side of \eqref{die-m}.
Moreover, we have the exponential mixing property of \eqref{die-m}:
for any $\tau \in (0, \tau_0)$, \eqref{die-m} (whose Markov semigroup is denoted by $\hat P_m$) possesses a unique invariant measure $\hat \pi_\tau$ in $H$ such that for any $m \in \nn$, $x \in H$, and $\phi \in {\rm Lip}_b(H)$, 
\begin{align*}
|\hat P_m \phi(x)-\hat \pi_\tau(\phi)| \le \exp(-\gamma_8 t_m/2)|\phi|_{\rm Lip}\int_H \|x-y\| \hat \pi_\tau ({\rm d} y).
\end{align*}    
 
\item
Using the method in \cite[Theorem 4.1]{Liu22} and the arguments in the proof of Theorem \ref{tm-err}, one can show a similar uniform strong error estimate for \eqref{die-m} and its Galerkin full discretization.
In particular, if $X_0\in \dot H^{1+\gamma}$ with some $\gamma\in [0,1]$, then the DIEM Galerkin scheme (the Galerkin scheme of \eqref{die-m}, denoted by $\{Y^h_k\}_{k \in \nn}$) possesses higher strong order of convergence in both spatial and temporal directions for any $p \ge 2$:
\begin{align*} 
& \sup_{k \in \nn} \ee \|X_{t_k}-Y^h_k\|^p
\le C (1+ \|X_0\|_{1+\gamma}^{p(q^2+q-1)}) (h^{1+\gamma}+\tau^\frac{1+\gamma}2)^p, 
\end{align*}  
provided more regular and growth conditions (see \cite[Assumptions 5.1 and 5.2]{LQ21}) on the coefficients of Eq. \eqref{see} hold true.  
\end{enumerate}
\end{rk}

\section{Applications to Stochastic Allen--Cahn Equation}
\label{sec4}

In the last section, let us apply the exponential ergodicity and uniform strong error analysis results, developed in Section \ref{sec3}, to the semiclassical stochastic Allen--Cahn equation \eqref{ac}. 

We note that as the driven noise in Eq. \eqref{ac} is additive, the DIE scheme \eqref{die} and DIEG scheme \eqref{die-g} coincide with the DIEM scheme \eqref{die-m} and its Galerkin full discretization, respectively. 
Therefore, in this part, it suffices to consider the DIE scheme \eqref{die} and DIEG scheme \eqref{die-g}.

Let us first check the conditions in the main Assumptions \ref{ap-f} and \ref{ap-g}  for Eq. \eqref{ac} with $f(x)=\alpha^{-2}(x-x^3)$ and $G \equiv {\rm Id}$.
At first, the growth order of $f'$ in \eqref{f'} is $q-1=2$, so that $q=3$.
Secondly, it is not difficult to show that $L_1=\alpha^{-2}$ in \eqref{f-mon} and $L_4=2\alpha^{-2}$ and $L_5=\alpha^{-2}$ in \eqref{con-f'}.
For \eqref{f-coe}, note that by Young inequality, 
\begin{align*}
\alpha^{-2}(x-x^3)x
=-\alpha^{-2} x^4+\alpha^{-2}x^2
\le -C x^2+C_\alpha,
\end{align*}     
for any positive $C$ and certain positive constant $C_\alpha$. 
So, one can take $L_2$ to be any negative scalar.
We emphasize that the choice of $L_2$ to be any negative scalar is only needed in this section to ensure the existence of the invariant measures of Eq. \eqref{ac} and its DIEG scheme \eqref{die-g}.

For the conditions on $G$ of Assumption \ref{ap-g}, $L_6=L_7=L_9=0$, $L_8={\rm Trace}({\bf Q})$, and $L_{10}=\|(-A)^{1/2} {\bf Q}^{1/2}\|^2_{HS(H;H)}$.
Therefore, $L_2+L_7/2=L_2<\lambda_1$ is always valid as one can take $L_2$ to be any negative scalar, so that the main conditions in Proposition \ref{prop-sta} and Lemma \ref{lm-sta-uk} hold and thus Eq. \eqref{ac} and its DIEG scheme \eqref{die-g} always possess at least one invariant measure, respectively.
Meanwhile, $L_1+L_6/2=\alpha^{-2}<\lambda_1$ for $\alpha>\lambda_1^{-1/2}$, so that the main conditions in Theorems \ref{tm-erg} and \ref{tm-erg-uk} hold and thus Eq. \eqref{ac} and its DIEG scheme \eqref{die-g} are both exponentially ergodic, respectively, provided that the interface thickness is not too small. 

Finally, when $\alpha>\lambda_1^{-1/2}$, we have $L_1+21 L_9/2 =\alpha^{-2} < \lambda_1$, so that the main condition in Theorem \ref{tm-err} hold.
Therefore, we have exponential ergodicity and uniform strong error estimate for the stochastic Allen--Cahn equation \eqref{ac}.

\begin{tm}  
Eq. \eqref{ac} as well as its DIE scheme \eqref{die} and DIEG scheme \eqref{die-g} always possess at least one invariant measure, respectively.
They are all exponentially ergodic, provided that $\alpha>\lambda_1^{-1/2}$.
Moreover, there exist a positive constant $C$ such that for any $\tau \in (0, 1)$,     
\begin{align*} 
& \sup_{k \in \nn} \ee \|X_{t_k}-X^h_k\|^2 
\le C (1+ \|X_0\|_1^{22}) (h^2+\tau).
\end{align*}  
\end{tm}

\begin{rk} \label{rk-solution}
The authors in \cite[Page 88]{CHS21} proposed a question whether the invariant measure of the DIE scheme \eqref{die}, applied to Eq. \eqref{see-fg} with polynomial drift of order 3 with negative leading coefficients (as in Remark \ref{rk-ex-f}), driven by additive noise, is unique.  
In this case, $L_1$ is just the bound $\sup_{\xi \in \rr} f'(x)$ and $L_6=0$.
Applying Theorem \ref{tm-erg-uk}, we conclude that, in this case, the DIE scheme \eqref{die} is exponentially ergodic and thus gives an affirmative answer to the question proposed in \cite{CHS21}, provided that  $\sup_{\xi \in \rr} f'(x)<\lambda_1$. 
We also note that this restriction can be removed in the non-degenerate noise case using the methodology developed in the recent paper \cite{LL24}.
\end{rk}

\begin{rk}
Combining the uniform ergodicity of the DIEG scheme \eqref{die-g} with the uniform weak error estimate in \cite{CHS21}, one can show the error estimate between the fully discrete invariant measure $\pi^h_\tau$ of \eqref{die-g} and the exact invariant measure $\pi$ of Eq. \eqref{ac}.
\end{rk}


\bibliographystyle{amsplain}
\bibliography{bib}

\end{document}